\documentclass[11pt, reqno, english]{amsart}  
\usepackage[utf8]{inputenc}
\usepackage[T1]{fontenc}
\usepackage{amsmath,amsthm}
\usepackage{amsfonts,amssymb}
\usepackage{url}
\usepackage{mathtools}  
\usepackage[colorlinks=true,urlcolor=blue,linkcolor=red,citecolor=magenta]{hyperref}
\usepackage{enumerate, paralist}
\usepackage{tikz-cd}
\usepackage[left=1in,right=1in,top=1in,bottom=1in]{geometry}
\numberwithin{equation}{section}

\theoremstyle{plain}
\newtheorem{theorem}{Theorem}[section]
\newtheorem{corollary}[theorem]{Corollary}
\newtheorem{proposition}[theorem]{Proposition}

\newtheorem{question}[theorem]{Question}

\theoremstyle{definition}

\newcommand{\R}{\mathbb{R}}

\newcommand{\Z}{\mathbb{Z}}
\newcommand{\F}{\mathcal{F}}
\newcommand{\KG}{\mathrm{KG}}

\begin{document}

\title[The Generalized Makeev Problem Revisited]{The Generalized Makeev Problem Revisited}



\author{Andres Mejia}
\address[AM]{Dept.\ Math.\, University of Pennsylvania, Philadelphia, PA 19104, USA}
\email{amejia83@sas.upenn.edu}

\author{Steven Simon}
\address[SS]{Dept.\ Math.\, Bard College, Annandale-on-Hudson, NY 12504, USA}
\email{ssimon@bard.edu} 

\author{Jialin Zhang}
\address[JZ]{Dept.\ Math. and Stat.\, Boston University, Boston, MA 02215, USA}
\email{zjleric@bu.edu}


\begin{abstract} 
\small 

Based on a result of Makeev, in 2012 Blagojevi\'c and Karasev proposed the following problem: given any positive integers $m$ and $1\leq \ell\leq k$, find the minimum dimension $d=\Delta(m;\ell/k)$ such that for any $m$ mass distributions on $\R^d$, there exist $k$ hyperplanes, any $\ell$ of which equipartition each mass. The $\ell=k$ case is a central question in geometric and topological combinatorics which remains open except for few values of $m$ and $k$. For $\ell< k$ and arbitrary $m$, we establish new upper bounds on $\Delta(m;\ell/k)$ when (1) $\ell=2$ and $k$ is arbitrary and (2) $\ell=3$ and $k=4$. When $\ell=k-1$ and $m+1$ is a power of two these bounds are nearly optimal and are exponentially smaller than the current best upper bounds when $\ell=k$. Similar remarks apply to our upper bounds when the hyperplanes are prescribed to be pairwise orthogonal. Lastly, we provide transversal extensions of our results along the lines recently established by Frick et al.: given $m$ families of compact convex sets in $\R^d$ such that no $2^\ell$ members of any family are pairwise disjoint, we show that every member of each family is pierced by the union of any $\ell$ of some collection of $k$ hyperplanes.
\end{abstract}

\date{}
\maketitle

\section{Introduction and Statement of Main Results}

\subsection{A Brief History of Hyperplane Mass Equipartition Problems}  Measure equipartition problems concern the following general situation: given a family of finite measures on $\R^d$ which are absolutely continuous with respect to Lebesgue measure (henceforth to be called \textit{masses}), one seeks a division of $\R^d$ by prescribed interior disjoint (and typically convex) regions so that each region contains an equal fraction of each total mass. Beginning with the classical Ham Sandwich theorem~\cite{BZ04} (any $d$ masses on $\R^d$ can be equipartitioned by the two half-spaces determined by a single hyperplane) via the Borsuk--Ulam theorem, in recent years these problems have come to occupy a central position in geometric combinatorics with results almost invariably obtained by the application of equivariant topological machinery (see, e.g. ~\cite{Ma08, Zi17} and the recent survey ~\cite{RS22}). 

Of the mass equipartition problems, none has been so well studied as that which seeks extensions of the Ham Sandwich theorem to the multiple hyperplane setting (see, e.g., ~\cite{Av84, BFHZ16, BFHZ18, Gr60, Ha66, Kl21, MLVZ06, Ra96, YDEP89, Zi15}). This question dates back to the work of Gr\"unbaum~\cite{Gr60} and Hadwiger~\cite{Ha66} and was first studied in its current form by Ramos~\cite{Ra96}. To set notation, let $\{H^+,H^-\}$ denote the closed half-spaces determined by a hyperplane $H$ in $\R^d$. Thus any collection of $k$ hyperplanes $H_1,\ldots, H_k$ in $\R^d$ determines $2^k$ (not necessarily distinct or non-empty) regions $H_1^\pm \cap \cdots \cap H_k^\pm$, and we say that the $k$ hyperplanes equipartition a collection of masses if the resulting regions do (in which case the hyperplanes are necessarily distinct, and in fact affinely independent). 

\begin{question}[Gr\"unbaum--Hadwiger--Ramos Problem] 
\label{quest:G-H-R} 
Let $m\geq 1$ and $k\geq 1$ be integers. What is the minimum dimension $d:=\Delta(m;k)$ such that any $m$ masses on $\R^d$ can be equipartitioned by $k$ hyperplanes? 
\end{question}

A degrees of freedom count due to Avis~\cite{Av84} and Ramos~\cite{Ra96} gives  \begin{equation}
\label{eqn:Ramos} 
\Delta(m;k)\geq \left\lceil\frac{m(2^k-1)}{k}\right\rceil,
\end{equation} and this lower bound  is conjectured  to be tight in all cases. Besides the $k=1$ case (i.e., the Ham Sandwich theorem), however, exact bounds on $\Delta(m;k)$, all of which match the conjectured value, are currently known in only the following cases:
\begin{itemize}
\item for $k=2$, for all $m=2^r+s$, where $r\geq 1$ and $s\in\{-1,0,1\}$~\cite{MLVZ06, BFHZ16, BFHZ18}, and 
\item for $k=3$, if $m\in \{1,2,4\}$~\cite{Ha66, BFHZ16, BFHZ18}.
\end{itemize}

 The best general upper bound is due to ~\cite{MLVZ06}. Writing $m=2^{q+1}-t$ for $q\geq 0$ and $1\leq t\leq 2^q$, one has 
\begin{equation}
\label{eqn:M-L-Z} \Delta(2^{q+1}-t;k)\leq 2^q\cdot (2^{k-1}+1)-t. 
\end{equation}
For multiple hyperplanes, this upper bound matches the Ramos lower bound precisely when $k=2$ and $m+1$ is a power of 2 and deviates exponentially from this bound in both parameters $k$ and $q$.

Inspired by a result of Makeev~\cite{Ma07} showing that a single mass in $\R^d$ can be equipartitioned by any pair of some collection of $d$ pairwise orthogonal hyperplanes, Blagojevi\'c and Karasev ~\cite{BK12} posed a fractional generalization of Question~\ref{quest:G-H-R}:

 \begin{question}[Generalized Makeev Problem] 
\label{quest:generalized Makeev}   
Let $m\geq 1$ and $1\leq \ell \leq k$ be integers. What is the minimum dimension $d:=\Delta(m;\ell/k)$ (respectively, $d=\Delta^\perp(m;\ell/k))$ such that for any $m$ masses on $\R^d$, there exist $k$ (pairwise orthogonal) hyperplanes, any $\ell$ of which equipartition each mass? 
 \end{question}

 The main result of ~\cite{BK12} gave polynomial criteria for upper bounds to Question~\ref{quest:generalized Makeev}, which, as with the upper bound ~(\ref{eqn:M-L-Z}), arise from ideal--valued cohomological index theory ~\cite{FH98} applied to the group $\mathbb{Z}_2^k$. In the non-orthogonal setting this technique resulted in the rough upper bound $\Delta(m;\ell/k)\leq m \left(\sum_{j=0}^\ell \binom{k-1}{j}\right)$, while for $\Delta^\perp(m;\ell/k)$ no upper bound was obtained; see Section~\ref{sec:polynomial} for more details. 
 
\subsection{New Upper Bounds for the Generalized Makeev Problem}

Our paper gives improved upper bounds for Question~\ref{quest:generalized Makeev} when $\ell=2$ and $k$ is arbitrary, as well as when $\ell=3$ and $k=4$:

\begin{theorem}
\label{thm:upper}
Let $q\geq 0$ and $1\leq t\leq 2^q$ be integers.
\begin{compactenum}[(a)]
\item $\Delta(2^{q+1}-t;2/k)\leq 2^q\cdot (k+1)-t$ and
\item $\Delta(2^{q+1}-t;3/4)\leq 7\cdot 2^q -2t$
\end{compactenum} 
\end{theorem}

 We observe that the upper bounds of ~\cite{BK12} above give $\Delta(2^{q+1}-t;2/k)\leq (2^{q+1}-t)\cdot (1+\binom{k}{2})$ and $\Delta(2^{q+1}-t;3/4)\leq 7\cdot (2^{q+1}-t)$. The upper bounds Theorem~\ref{thm:upper} are nearly tight when $m+1$ is a power of two and $\ell=k-1$:

\begin{corollary}
\label{cor:explicit} Let $q\geq 0$ be an integer. \begin{compactenum}[(a)]
\item $4\cdot 2^q-2\leq \Delta(2^{q+1}-1;2/3) \leq 4\cdot 2^q-1$ 
\item $7\cdot 2^q-3\leq \Delta(2^{q+1}-1;3/4) \leq 7\cdot 2^q-2$
\end{compactenum}
\end{corollary}

For context, we compare Corollary~\ref{cor:explicit} to the current known results for $\Delta(2^{q+1}-1;k)$ and $k\in\{3,4\}$. While $\Delta(1;2/3)\leq 3$ and $\Delta(1;3/4)\leq 5$ are immediate consequences of  $\Delta(1;3)=3$~\cite{Ha66} and $4\leq \Delta(1;4)\leq 5$~\cite{BFHZ18}, respectively, the best upper bounds for $q\geq 1$ remain $\Delta(2^{q+1}-1;3)\leq 5\cdot 2^q -1$ and $\Delta(2^{q+1}-1;4)\leq 9\cdot 2^q -1$ given by (\ref{eqn:M-L-Z}). Thus the upper bounds of Theorem~\ref{cor:explicit} represents an exponential improvement over those for Question~\ref{quest:G-H-R}, despite there being one degree of freedom more per mass. For example, Theorem~\ref{cor:explicit} gives $11\leq \Delta(3;3/4)\leq 12$ while $12\leq \Delta(3;4)\leq 17$.

 While Theorem~\ref{cor:explicit} ultimately relies on standard topological schemes used previously, we are nonetheless able to obtain new upper bounds to Question~\ref{quest:generalized Makeev} by a careful analysis of the corresponding polynomials.  In particular, this allows us to impose further constraints on the equipartitioning hyperplanes in the same dimensions given by our upper bounds; see~\cite{Si19} for similar strengthening of the bound~(\ref{eqn:M-L-Z}) in the context of Question~\ref{quest:G-H-R}. For example, when $q\geq 1$ and $2\leq t\leq 2^q$ one may guarantee pairwise orthogonality of the equipartitioning hyperplanes, thereby yielding explicit upper bounds for the orthogonal version of Question~\ref{quest:generalized Makeev}:

\begin{theorem} 
\label{thm:upper perp} 
Let $q\geq 0$ and $1\leq t\leq 2^q$ be integers.
\begin{compactenum}[(a)]
\item $\Delta^\perp(2^{q+1}-t;2/k)\leq 2^q\cdot (k+1) -t$ and
\item $\Delta^\perp(2^{q+1}-t;3/4)\leq 7\cdot 2^q -2t$
\end{compactenum} 
\end{theorem}

As with Theorem~\ref{thm:upper}, these results are nearly tight when $t=2$ and $(\ell,k)\in \{(2,3),(3,4)\}$. 

\begin{corollary}
\label{cor:explicit perp} Let $q\geq 1$ be an integer.
\begin{compactenum}[(a)]
\item $4\cdot 2^q-3\leq \Delta^\perp(2^{q+1}-2;2/3)\leq 4\cdot 2^q -2$ 
\item $7\cdot 2^q-5\leq \Delta^\perp(2^{q+1}-2;3/4)\leq 7\cdot 2^q -4$
\end{compactenum}
\end{corollary}

For example, any two masses in $\R^d$ can be equipartitioned by any two of some collection of three pairwise orthogonal hyperplanes when $d=6$, and by any three of some collection of four pairwise orthogonal hyperplanes when $d=10$. These results can be compared to $\Delta^\perp(2;2)=4$~ and $6\leq \Delta^\perp(2;3)\leq 8$ of ~\cite{Si19}, respectively. Separately, in the Appendix we establish the estimates $$5 \leq \Delta^\perp(1;3/4)\leq 7\,\, \text{and} \,\, 7\leq \Delta^\perp(1;3/5)\leq 9 $$ which can be compared to the known results $\Delta^\perp(1;3)=4$ and $6\leq \Delta^\perp(1;4)\leq 8$ of~\cite{Si19}.

\subsection{Transversal Extensions} A theorem of Dolnikov~\cite{Do92} states that if $\mathcal{F}_1,\ldots, \mathcal{F}_d$ are intersecting finite families of compact convex sets in $\R^d$ (i.e., the intersection of any two members of each family is non-empty), then there exists a hyperplane in $\R^d$ which intersects every member of each family. Such a  hyperplane is said to \textit{pierce} each $\F_i$. This result was shown to imply a discrete (and equivalent) version of the Ham Sandwich theorem: for any $d$ finite point sets in $\R^d$, there exists a hyperplane whose \textit{open} half-spaces contain no more than half the points from each set.

In ~\cite{FMSS22}, Dolnikov's theorem was generalized to multiple hyperplanes. Given a finite family $\F$ of compact convex sets in $\R^d$, one says that a collection of hyperplanes pierces $\F$ provided their union intersects each member of $\F$. For the three infinite families $m=2^r+s$ for which $d=\Delta(m;2)=\left\lceil\frac{3m}{2}\right\rceil$ is known, it was shown that any $m$ families of compact convex sets in $\mathbb{R}^d$ can be pierced by two hyperplanes so long as no \textit{four} members of any family were pairwise disjoint. As with Dolnikov's theorem, one recovered equivalent discrete versions of the mass equipartition results as special cases. Here we show the analogous transversal extensions of Theorem~\ref{thm:upper}:

\begin{theorem}
\label{thm:transversal} 
Let $q\geq 0$ and $1\leq t\leq 2^q$ be integers, let $m=2^{q+1}-t$, and let $\F_1,\ldots, \F_m$ be finite families of compact convex sets in $\R^d$.
\begin{compactenum}[(a)]
\item Let $d=2^q\cdot (k+1)-t$. If no four members of any $\F_i$ are pairwise disjoint, then every $\F_i$ can be pierced by any two of some collection of $k$ hyperplanes in $\R^d$.
\item If $d=7\cdot 2^q -2t$ and no eight members of any $\F_i$ are pairwise disjoint, then every $\F_i$ can be pierced by any three of some collection of four hyperplanes in $\R^d$. 
\end{compactenum}
\end{theorem} 

As an example, Theorem~\ref{thm:transversal} gives the following generalization of $\Delta(1;3/4)\leq 5$: any finite family of compact convex sets in $\R^5$ can be pierced by any three of some collection of four hyperplanes provided no eight members of the family are pairwise disjoint. As we have already observed, $\Delta(1;3/4)\leq 5$ is an immediate consequence of $\Delta(1;4)\leq 5$. To our knowledge, however, there is no transversal generalization of the latter.

\section{Configuration-Spaces and Test Maps for Question~\ref{quest:generalized Makeev}}
\label{sec:topology} 

We begin by recalling the established configuration-space/test-map paradigm by which hyperplane mass partition problems are reduced to equivariant topology. 
\subsection{Configuration-Spaces} 

The starting point for applications of equivariant topology to hyperplane mass partition problems is the identification of pairs of half-spaces in $\R^d$ with pairs of antipodal points on the unit sphere $S^d$ in $\R^{d+1}$. Explicitly, for each $x=(\mathbf{a},b)\in \R^d\times \R$ with $\|\mathbf{a}\|^2+|b|^2=1$, one defines $H(x)=\{\mathbf{u}\in \R^d\mid \langle \mathbf{u}, \mathbf{a}\rangle =b\}$ and $$H^+(x)=\{\mathbf{u}\in \R^d\mid \langle \mathbf{u},\mathbf{a}\rangle \geq b\}\,\, \text{and}\,\, H^-(x)=\{\mathbf{u}\in \R^d\mid \langle \mathbf{u},\mathbf{a}\rangle \leq b\}.$$
\noindent If $x\neq (\mathbf{0},\pm 1)$, then the sets $H^+(x)$ and $H^-(x)$ are the closed half-spaces determined by the hyperplane $H(x)$. As $H(x)=H(-x)$ and $H^\pm(-x)=H^\mp(x)$ for all $x\in S^d$, the antipodal $\mathbb{Z}_2$-action on the sphere coincides with the reflective $\mathbb{Z}_2$-action on pairs of half-spaces. If $x=(\mathbf{0},\pm 1)$, then $H(x)$ is a ``hyperplane at infinity'', in which case $\{H^+(x), H^-(x)\}=\{\emptyset, \R^d\}$. 

Under this identification, the regions determined by any arrangement of $k\geq 2$ hyperplanes can be parametrized by the $k$-fold product $(S^d)^k$ of spheres. Namely, let $\mathbb{Z}_2=\{0,1\}$, and set $H^0(x)=H^+(x)$ and $H^1(x)=H^-(x)$ for each $x\in S^d$. We let $$\mathcal{R}_g(x)=H^{g_1}(x_1)\cap \cdots \cap H^{g_k}(x_k)$$ for each $x=(x_1,\ldots, x_k)\in (S^d)^k$ and each $g=(g_1,\ldots, g_k)\in \mathbb{Z}_2^k$. As before, $\mathcal{R}_{g}(h\cdot x)=\mathcal{R}_{gh}(x)$ for each $g,h\in \mathbb{Z}_2^k$ and $x\in (S^d)^k$, so the standard $\mathbb{Z}_2^k$-action on $(S^d)^k$ corresponds to the reflective action on each collection $\{\mathcal{R}_g(x)\}_{g\in \mathbb{Z}_2^{k}}$. A priori, some of the hyperplanes may be at infinity and not all of the regions $\mathcal{R}_g(x)$ are necessarily distinct. If any $2\leq \ell\leq k$ of the $k$ hyperplanes are to equipartition a given mass $\mu$, however, then none of the hyperplanes are at infinity and the hyperplanes are distinct (morever, any $\ell$ of the hyperplanes are affinely independent).

Considering the action of the symmetric group $\mathfrak{S}_k$ on the coordinate spheres of $(S^d)^k$ by permutations, the above identification coincides with the $\mathfrak{S}_k$-action on any collection of $k$ hyperplanes, and so the natural action of the hyperoctahedral group $\mathfrak{S}_k^\pm:=\mathbb{Z}_2^k \rtimes \mathfrak{S}_k$ on $(S^d)^k$ coincides with that on each collection $\{\mathcal{R}_g(x)\}_{g\in \mathbb{Z}_2^k}$. While the use of the full $\mathfrak{S}_k^\pm$-action was essential in the derivation of all known tight $k\in\{2,3\}$ results to Question~\ref{quest:G-H-R} except those given by the upper bound ~(\ref{eqn:M-L-Z}), in what follows we shall \textit{only} need the action of the subgroup $\mathbb{Z}_2^k$, in which case the action on $(S^d)^k$ is free. 

\subsection{Test Spaces and Test Maps} Let $1\leq \ell \leq k$, and suppose that $\mu_1,\ldots, \mu_m$ are masses on $\R^d$ which we seek to equipartition by any $\ell$ of some collection of $k$ hyperplanes. We shall adopt the ``Fourier perspective'' for mass equipartitions introduced in~\cite{Si15}. To each decomposition $(\mathcal{R}_g)_{g\in \mathbb{Z}_2^k}$  of $\R^d$ determined by $k$ hyperplanes $H_1,\ldots, H_k$ (some possibly at infinity, not all necessarily distinct), let $f_i \colon \mathbb{Z}_2^k \rightarrow \R$ be given by $g\mapsto \mu_i(\mathcal{R}_g)$ for each $g\in\mathbb{Z}_2^k$. The Fourier expansion of each such map is 

$$f_i=\sum_{h\in \mathbb{Z}_2^{k}}c_h (f_i) \chi_h,$$ where for each $h=(h_1,\ldots, h_k)\in \mathbb{Z}_2^k$ the $\mathbb{Z}_2^k$-representation $\chi_h\colon \mathbb{Z}_2^{k}\rightarrow \{\pm 1\}$ is defined by $\chi_h(g)=(-1)^{h_1g_1+\cdots + h_kg_k}$ for each $g=(g_1,\ldots, g_k)$ in $\mathbb{Z}_2^k$ and whose corresponding Fourier coefficients are $$c_h(f_i)=\frac{1}{2^k}\sum_{g\in \mathbb{Z}_2^k} f_i(g)\chi_h(g).$$

We recall that the set of all $\chi_g$ forms a basis for the space of all functions $f\colon \mathbb{Z}_2^k\rightarrow \R$ (and in fact an orthonormal basis under the inner-product $\langle f_1, f_2\rangle =\frac{1}{2^k} \sum_{g\in G} f_1(g)f_2(g)$; see, e.g.~\cite{Se77}). An equipartition of the $\mu_i$ by any $\ell$ of $k$ hyperplanes can be characterized in terms of Fourier coefficients as follows: 

\begin{proposition}
\label{prop:Fourier}
Let $\mu_1,\ldots, \mu_m$ be masses on $\R^d$ and let  $H_1,\ldots, H_k$ be hyperplanes on $\R^d$. Any $\ell$ of the $k$ hyperplanes equipartitions each $\mu_i$ if and only if $c_h(f_i)=0$ for all $h\in E_{\ell,k}:=\{(h_1,\ldots, h_k)\in \mathbb{Z}_2^k\mid 1\leq h_1+\cdots + h_k\leq \ell\}$ and all $1\leq i \leq m$.
\end{proposition}

\begin{proof}[Proof of Proposition~\ref{prop:Fourier}]
For each choice of $1\leq i_1<i_2<\cdots < i_\ell\leq k$, let $\mathbb{Z}_2^{\ell}:=\langle \mathbf{e}_{i_1},\ldots, \mathbf{e}_{i_\ell}\rangle$ be the subgroup of $\mathbb{Z}_2^k$ generated by the corresponding standard basis vectors of $\mathbb{Z}_2^k$ and let $\mathbb{Z}_2^{k-\ell}$ be the subgroup generated by the remaining basis vectors. The regions $\{\mathcal{R}_g^\ell\}_{g\in \mathbb{Z}_2^\ell}$ determined by the hyperplanes $H_{i_1},\ldots, H_{i_\ell}$ are given by $\mathcal{R}_g^\ell=\bigcup_{g'\in \mathbb{Z}_2^{k-\ell}} \mathcal{R}_{g+g'}$ for each $g\in \mathbb{Z}_2^\ell$, so $\mu(\mathcal{R}_g^\ell)=\sum_{g'\in \mathbb{Z}_2^{k-\ell}}\mu(\mathcal{R}_{g+g'})$ for each $g\in \mathbb{Z}_2^\ell$. Observe also that $\chi_h(g)=\chi_h(g)\chi_h(g')=\chi_h(g+g')=$ for each $g,h\in \mathbb{Z}_2^\ell$ and $g'\in \mathbb{Z}_2^{k-\ell}$. Consider the maps $f_i^\ell\colon \mathbb{Z}_2^\ell\rightarrow \R$ given by $f_i^\ell(g)=\mu(\mathcal{R}_g^\ell)$. These have  Fourier expansions $$f_i^\ell=\sum_{h\in \mathbb{Z}_2^\ell} c_h(f_i^\ell)\chi_h{_{\mid \mathbb{Z}_2^\ell}},$$ and by the above observations we see that $$c_h(f_i^\ell)=
\frac{1}{2^\ell}\sum_{g\in \mathbb{Z}_2^\ell}\sum_{g'\in \mathbb{Z}_2^{k-\ell}}f_i(g+g')\chi_h{_{\mid \mathbb{Z}_2^\ell}}(g)=\frac{1}{2^\ell}\sum_{g\in \mathbb{Z}_2^\ell} \sum_{g'\in \mathbb{Z}_2^{k-\ell}}f_i(g+g')\chi_h(g+g')=2^{k-\ell}c_h(f_i)$$ for each $h\in \mathbb{Z}_2^\ell$. On the other hand, $\chi_0=1$ and therefore $c_0(f_i^\ell)=\frac{1}{2^\ell}\sum_{g\in \mathbb{Z}_2^\ell}\mu_i(\mathcal{R}_g)=\frac{1}{2^\ell} \mu_i(\R^d)$. Examining the Fourier decomposition of the $f_i^\ell$, it follows that $H_{i_1},\ldots, H_{i_\ell}$ equipartition $\mu_i$ if and  only if $c_h(f_i^\ell)=0$ for all $h\in \mathbb{Z}_2^\ell\setminus \{0\}$, and therefore if and only if $c_h(f_i)=0$ for all such $h$. Ranging over all choices of any $\ell$ of the $k$ hyperplanes, one has that any $\ell$ of the hyperplanes equipartitions $\mu_i$ if and only if $c_h(f_i)=0$ for each $h\in E_{\ell,k}$. \end{proof}

For each $h\in \mathbb{Z}_2^k$, let $V_h=\R$ denote the 1-dimensional $\mathbb{Z}_2^k$-module determined by the action of the representation $\chi_h\colon \mathbb{Z}_2^k\rightarrow \{\pm 1\}$. For each $x\in (S^d)^k$ and any $1\leq i \leq k$, evaluating Fourier coefficients  determines a continuous map $\F_{i,h}\colon (S^d)^k\rightarrow V_h$ given by \begin{equation} \F_{i,h}(x)=\frac{1} {2^k}\sum_{g\in \mathbb{Z}_2^k}\mu_i(\mathcal{R}_g(x))\chi_h(g).\end{equation} 
Letting $U_{\ell,k}=\oplus_{h\in E_{\ell,k}} V_h$, Proposition~\ref{prop:Fourier} shows that zeros of the test-map  
\begin{equation}
\label{eqn:Fourier Test} 
\F=\oplus_{i=1}^m \oplus_{h\in E_{\ell,k}} \F_{i,h}\colon  (S^d)^k\rightarrow U_{\ell,k}^{\oplus m}\end{equation} correspond bijectively to families of $k$ hyperplanes in $\R^d$, any $\ell$ of which equipartitions each $\mu_i$. In determining upper bounds to Question~\ref{quest:generalized Makeev}, it will be crucial to observe that each map $\F_{i,h}\colon (S^d)^k\rightarrow V_h$ above is  $\mathbb{Z}_2^k$-equivariant, and therefore that $\F$ is $\mathbb{Z}_2^k$-equivariant as well. 

In order to incorporate pairwise orthogonality of the hyperplanes, we proceed as in~\cite{Si19} and impose the condition on the target as opposed to the configuration space. For each $1\leq r<s\leq k$, let $\pi_{r,s}\colon (S^d)^k\rightarrow V_{\mathbf{e}_r+\mathbf{e}_s}$ be given by $\pi_{r,s}(x)=\langle q(x_r),q(x_s)\rangle$, where $q\colon S^d\rightarrow \R^d$ is the projection of $S^d$ onto its first $d$ coordinates. Each such map is continuous and $\mathbb{Z}_2^k$-equivariant, so $$\pi=\oplus_{1\leq r<s\leq k}\pi_{r,s}\colon (S^d)^k\rightarrow U(\mathcal{O})\colon=  \oplus_{1\leq r<s\leq k} V_{\mathbf{e}_r+\mathbf{e}_s}$$ is equivariant as well. Any zero of the enlarged test map \begin{equation}
    \label{eqn:orthogonality}
\mathcal{F}\oplus \pi\colon (S^d)^k\rightarrow U_{\ell,k}^{\oplus m}\oplus U(\mathcal{O})
\end{equation} is now seen to correspond to a collection of $k$ pairwise orthogonal hyperplanes, any $\ell$ of which equipartition each mass.

\subsection{Lower Bounds}

By the orthogonality of characters, the $|E_{\ell,k}|=\sum_{j=1}^\ell \binom{k}{j}$ equipartition conditions on any mass are independent. A degrees of freedom count applied to the map $\mathcal{F}$ therefore leads to an expected lower bound of $k\Delta(m;\ell/k)\geq m\sum_{j=1}^\ell \binom{k}{j}$, and likewise $k\Delta^\perp (m;\ell/k)\geq m\sum_{j=1}^\ell \binom{k}{j} + \binom{k}{2}$ in the orthogonal setting. The following weaker lower bound matches that expected when $\ell\in\{k-1,k\}$. In particular, one recovers the Ramos lower bound~(\ref{eqn:Ramos}) and the lower bounds of Corollaries~\ref{cor:explicit} and ~\ref{cor:explicit perp} and Proposition~\ref{prop:perp} of the Appendix. 

\begin{proposition}  \label{prop:lower}

Let $m\geq 1$ and $2\leq \ell \leq k$ be integers. Then $k\Delta(m;\ell/k)\geq m(2^\ell-1)(k-\ell+1)$ and $k\Delta^\perp (m;\ell/k)\geq  m(2^\ell-1)(k-\ell+1)+ \binom{k}{2}$. 
 \end{proposition}

\begin{proof}[Proof of Proposition~\ref{prop:lower}] 

For each $1\leq i \leq m$, let $C_i=\{p_{i,j}\mid 1\leq j\leq  2^\ell -1\}$ be a collection of $2^\ell-1$ points in $\R^d$ such that $\cup_{i=1}^m C_i$ is in general position. For each $i$ and $j$, let $B_\epsilon(p_{i,j})$ be the closed $\epsilon$-ball centered at $p_{i,j}$. By choosing $\epsilon>0$ to be sufficiently small, the balls are pairwise disjoint and any hyperplane intersects at most $d$ of the balls. 

For each $1\leq i \leq m$, let $\mu_i$ be the normalization of Lebesgue measure on $\R^d$ restricted to $\cup_{j=1}^{2^\ell-1} B_\epsilon(p_{i,j})$. Supposing that $d\leq \Delta(m;\ell/k)$, there are $k$ hyperplanes in $\R^d$, any $\ell$ of which equipartitoins each $\mu_i$. As $\mu_i(B_\epsilon(p_{i,j}))=\frac{1}{2^\ell-1}<\frac{1}{2^\ell}$, each ball is intersected by any $\ell$ of the $k$ hyperplanes and therefore by at least $k-\ell-1$ hyperplanes in all. As the hyperplanes are distinct and the balls are disjoint, this gives at least $m(2^\ell-1)(k-\ell-1)$ distinct points from $\cup_{i,j} B_\epsilon(p_{i,j})$ which lie on the union of the $k$ hyperplanes. On the other hand, each hyperplane intersects at most $d$ of the balls, so $kd\geq m(2^\ell-1)(k-\ell+1)$.  

In the orthogonal setting, we choose the points $p_{i,j}$ in general position as before, and moreover so that no more than $kd-\binom{k}{2}$ points lie on the union of any collection of $k$ pairwise orthogonal hyperplanes. Again, one chooses $\epsilon>0$ so that the $B_\epsilon(p_{i,j})$ are disjoint, no hyperplane intersects any $d$ balls, and the union of $k$ pairwise orthogonal hyperplanes intersects at most $kd-\binom{k}{2}$ balls. If $d\leq \Delta^\perp(m;\ell/k)$, then one has $k$ pairwise orthgonal hyperplanes in $\R^d$, the union of which contains at least $m(2^\ell-1)(k-\ell+1)$ distinct points lying on $\cup_{i,j} B_\epsilon(p_{i,j})$. Thus $kd-\binom{k}{2}\geq m(2^\ell-1)(k-\ell+1)$.
\end{proof}

\subsection{Upper Bounds} 
\label{sec:polynomial} 

Given the equivariance of the test-maps $\mathcal{F}$, upper bounds $\Delta(m;\ell/k)\leq d$ follow provided that arbitrary continuous $\mathbb{Z}_2^k$-equivariant maps $F\colon (S^d)^k\rightarrow U_{\ell,k}^{\oplus m}$ have a zero. In~\cite{BK12}, the following Borsuk-Ulam result was given in this context, whose polynomial criterion is an application of equivariant cohomological index theory for the group $\mathbb{Z}_2^k$ as given in~\cite{FH98}.

\begin{theorem} Let $p_{\ell,k}=\prod_{1\leq a_1+\ldots +a_k\leq \ell} (a_1t_1+\cdots +a_kt_k)\in \mathbb{Z}_2[t_1,\ldots, t_k].$ If $p_{\ell,k}^m$ $\notin \langle t_1^{d+1},\ldots, t_k^{d+k}\rangle$, then any continuous $\mathbb{Z}_2^k$-equivariant map $F\colon (S^d)^k\rightarrow U_{k,\ell}^{\oplus m}$ has a zero.
\end{theorem}

The upper bound $\Delta(m;\ell/k)\leq  m\left(\sum_{j=0}^\ell \binom{k-1}{j}\right)$ mentioned in the introduction follows by considering the maximal degree of $p_{\ell,k}$ in any variable. For our improved upper bounds for $\Delta(m;\ell/k)$ when $\ell=2$ or $\ell=3$ and $k=4$, we show that for minimum $d$ such that $p_{\ell,k}^m\notin \langle t_1^{d+1},\ldots, t_k^{d+1}\rangle$, there is a polynomial $q$ corresponding to additional constraints on the hyperplanes such that $p_{\ell,k}^m \cdot q=t_1^d\cdots t^d$ in $\mathbb{Z}_2[t_1,\ldots, t_k]/(t_1^d,\ldots, t_k^d)$. For this, we appeal to a  Borsuk-Ulam result for arbitrary $\mathbb{Z}_2^k$-representations stated for instance in ~\cite{Si19}. As before, for any $h\in \mathbb{Z}_2^k$ we denote by $V_h=\R$ the 1-dimensional real $\mathbb{Z}_2^k$-representation defined by $g\cdot v=(1)^{\langle g, h\rangle} v$ for $g\in \mathbb{Z}_2^k$ and $v\in \R$. 

\begin{proposition}
\label{prop:Borsuk-Ulam}
Let $A=(a_{i,j})_{1\leq i\leq kd,\,1\leq j\leq k}$ be a $(kd\times k)$-matrix with $\mathbb{Z}_2$-coefficients, let $U=\oplus_{i=1}^{kd} V_{(a_{i,1},\ldots, a_{i,k})}$, and let $p_U=\prod_{i=1}^{kd} (a_{i,1}t_1+\cdots +a_{i,k}t_k)\in \mathbb{Z}_2[t_1,\ldots, t_k]/(t_1^{d+1},\ldots, t_k^{d+1})$. If $p_U= t_1^d\cdots t_k^d$, then any continuous $\mathbb{Z}_2^k$-equivariant map $F\colon (S^d)^k \rightarrow U$ has a zero. 
\end{proposition}

\section{The Generalized Makeev Problem for $\ell=2$}
\label{sec:2}

In this section we consider the $\ell=2$ case of Question~\ref{quest:generalized Makeev} for arbitrary $k$. The following yields part (a) of Theorem~\ref{thm:upper} and in particular parts (a)--(b) of Theorem~\ref{cor:explicit}.

\begin{theorem}
\label{thm:2}
Let $k\geq 2$, $q\geq 0$, and $1\leq t \leq 2^q$ be integers and $d=2^q\cdot (k+1)-t$. For any $2^{q+1}-1$ masses on $\mathbb{R}^d$, there exists $k$ hyperplanes $H_1,\ldots, H_k$ such that \begin{compactenum}[(i)]
 \item any two of $H_1,\ldots, H_k$ equipartition $2^{q+1}-t$ of the masses,
  \item any two of $H_2,\ldots, H_k$ equipartition the remaining $t-1$ masses. 

    \end{compactenum} 
 \end{theorem} 

 For instance, given three masses on $\R^6$, Theorem~\ref{thm:2} gives three hyperplanes, any two of which equipartition two of the masses and two of which equipartition the third. This can be compared to the known results $\Delta(3;2)=5$~\cite{BFHZ18} and  $7\leq \Delta(3;3)\leq 9$~\cite{MLVZ06} for Question~\ref{quest:G-H-R}.

We may use the $\binom{k}{2}$ remaining degrees of freedom in Theorem~\ref{thm:2} to ensure that all but $k-1$ of the pairs of hyperplanes are pairwise orthogonal, and in addition that each of $H_2,\ldots, H_k$ bisects any prescribed further mass. When $2\leq t\leq 2^q$, one may ensure pairwise orthogonality of the hyperplanes by removing a mass.
\begin{theorem}
\label{thm:2 orthogonal} 
Let $k\geq 2$, $q\geq 1$, and $2\leq t \leq 2^q$ be integers and let $d=2^q\cdot (k+1)-t$. For any $2^{q+1}-2$ masses on $\mathbb{R}^d$, there exist $k$ pairwise orthogonal hyperplanes $H_1,\ldots, H_k$ such that 
\begin{compactenum}[(i)]
\item any two of $H_1,\ldots, H_k$ equipartition $2^{q+1}-t$ of the masses,
\item any two of $H_2,\ldots, H_k$ equipartition the remaining $t-2$ masses.\end{compactenum}

\end{theorem}

To exhaust the remaining degrees of freedom in Theorem~\ref{thm:2 orthogonal}, we specify that for any additional $k-1$ masses $\mu_1,\ldots, \mu_k$ on $\R^d$ one has that each hyperplane of $\{H_{i+1},\ldots, H_k\}$ bisects $\mu_i$ for all $1\leq i \leq k-1$.

\subsection{Proof of Theorems~\ref{thm:2} and ~\ref{thm:2 orthogonal}} 

We first set some notation. As before, let $$p_{\ell,k}= \prod_{1\leq a_1+\cdots + a_k\leq \ell} (a_1t_1+\cdots +a_kt_k).$$ Thus $p_{\ell,k}=\prod_{j=1}^\ell r_{j,k},$ where
$$r_{j,k}=\prod_{a_1+\cdots+a_k=j}(a_1t_1+\cdots +a_kt_k).$$ When there is no risk of ambiguity, for $i\geq 2$ we shall let $p_{\ell,k-i+1}=\prod_{1\leq a_i+\cdots + a_k\leq \ell} (a_it_i+\cdots +a_kt_k)$ denote the polynomial corresponding to the equipartition of a single mass by any $\ell$ of the hyperplanes $H_i,\ldots, H_k$, and likewise for $r_{j,k-i+1}$. 

For orthogonality, let $\mathcal{S}\subseteq \mathcal{O}=\{(i,j)\mid 1\leq i<j\leq k\}$ and $U(\mathcal{S})=\oplus_{(i,j)\in S} V_{\mathbf{e}_i+\mathbf{e}_j}$ and set $$p_{\mathcal{S}}=\prod_{(i,j)\in \mathcal{S}}(a_it_i+a_jt_j).$$ We observe that $$p_{\mathcal{O}}=r_{2,k}=\sum_{\sigma\in \mathfrak{S}_k}t_{\sigma(1)}^{k-1}t_{\sigma(2)}^{k-2}\cdots t_{\sigma(k)}^0$$ is the classical Vandermonde determinant, so $$p_{2,k}=r_{1,k}\cdot r_{2,k}=\sum_{\sigma\in \mathfrak{S}_k}t_{\sigma(1)}^kt_{\sigma(2)}^{k-1}\cdots t_{\sigma(k)}.$$

\begin{proof}[Proof of Theorem~\ref{thm:2}]
 Let $m=2^{q+1}-t$, where $1\leq t\leq 2^q$, and let $d=2^q\cdot (k+1)-t$. We let $\mathcal{S}=\{(i,j)\in \mathcal{O}\mid i\neq j+1\}$. By the configuration-space/test-map scheme of Section ~\ref{sec:topology},  Theorem~\ref{thm:2} follows provided any continuous $F\colon (S^d)^k\rightarrow U\colon=U_{2,k}^{\oplus m}\oplus U_{2,k-1}^{\oplus (t-1)}\oplus U(\mathcal{S})\oplus U_{1,k-1}$ has a zero. 
 
We have $p_{2,k}^m=p_{2,k}^{2^q}\cdot p_{2,k}^{2^q-t}$, and since the coefficients are in $\mathbb{Z}_2$ this gives \begin{align*} p_U=
 \sum_{\sigma\in \mathfrak{S}_k}t_{\sigma(1)}^{2^q k}t_{\sigma(2)}^{2^q(k-1)} \cdots t_{\sigma(k)}^{2^q}\cdot (\sum_{\tau\in \mathfrak{S}_k}t_{\tau(1)}^kt_{\tau(2)}^{k-1}\cdots t_{\tau(k)})^{2^q-t}\cdot (\sum_{\phi\in \mathfrak{S}_{k-1}}t_{\phi(2)}^{k-1}\cdots t_{\phi(k)})^{t-1}\\ \cdot (t_1+t_3)\cdots (t_1+t_k)\cdots (t_{k-2}+x_k) \cdot t_2\cdots t_k.\end{align*} We must have $\sigma(1)=1$ by exponent considerations. Indeed, supposing that $\sigma(1)=i$ for some $i\geq 2$ and letting $\deg(t_i)$ denote the exponent of $t_i$ gives $\deg(t_i)\geq 2^q\cdot k +(2^q-t)+1>d$. We must also have $\tau(k)=1$ in each of the sums of the second product, since again $\deg(t_1)>d$ otherwise. As any further contribution of $t_1$ also results in $\deg(t_1)>d$, we have $t_1^d(t_1+t_3)\cdots (t_1+t_k)=t_1^dt_3\cdots t_k$. Thus  
\begin{align*}p_U=t_1^d
\sum_{\sigma\in \mathfrak{S}_{k-1}}t_{\sigma(2)}^{2^q(k-1)}t_{\sigma(3)}^{2^q(k-2)} \cdots t_{\sigma(k)}^{2^q}\cdot (\sum_{\tau\in \mathfrak{S}_{k-1}}t_{\tau(2)}^k
t_{\tau(3)}^{k-1}\cdots t_{\tau(k)}^2)^{2^q-t}\cdot (\sum_{\phi\in \mathcal{S}_{k-1}}t_{\phi(2)}^{k-1}\cdots t_{\phi(k)})^{t-1}\\ \cdot (t_2+t_4)\cdots (t_2+t_k)\cdots (t_{k-2}+t_k) \cdot t_2t_3^2\cdots t_k^2.\end{align*}
\noindent Continuing in this fashion, exponent considerations give $\sigma(2)=\tau(k)=\phi(k)=2$ for all $\tau$ and $\phi$, and $t_2^d(t_2+t_4)\cdots (t_2+t_k)=t_2^dt_4\cdots t_k.$ Thus \begin{align*}p_U=t_1^dt_2^d\sum_{\sigma\in \mathfrak{S}_{k-2}}t_{\sigma(3)}^{2^q(k-2)}\cdots t_{\sigma(k)}^{2^q}\cdot (\sum_{\tau\in \mathfrak{S}_{k-2}}t_{\tau(3)}^kt_{\tau(3)}^{k-1}\cdots t_{\tau(k)}^3)^{2^q-t}\cdot (\sum_{\phi\in \mathfrak{S}_{k-2}}t_{\phi(3)}^{k-1}\cdots t_{\phi(k)}^2)^{t-1}\\ \cdot (t_3+t_5)\cdots (t_3+t_k)\cdots (t_{k-2}+t_k) \cdot t_3^2t_4^3\cdots t_k^3.\end{align*}

\noindent Continuing inductively yields $p_U=t_1^d\cdots t_k^d$, so $F$ has a zero by Proposition~\ref{prop:Borsuk-Ulam}.\end{proof}

\begin{proof}[Proof of Theorem~\ref{thm:2 orthogonal}]
Let $m=2^{q+1}-t$ with $2\leq t\leq 2^q$, and let $d=2^q(k+1)-t$. As before, we show that any $\mathbb{Z}_2^k$-equivariant map $F\colon (S^d)^k\rightarrow U\colon= U_{2,k}^m\oplus U_{2,k-1}^{\oplus(t-2)}\oplus U(\mathcal{O})\oplus U_{1,k-1}\oplus U_{2,k-2}\oplus \cdots \oplus U_{1,1}$ has a zero.  We have

\begin{align*}
p_U= \sum_{\sigma\in \mathfrak{S}_k}t_{\sigma(1)}^{2^q k}t_{\sigma(2)}^{2^q(k-1)} \cdots t_{\sigma(k)}^{2^q}\cdot (\sum_{\tau\in \mathfrak{S}_k}t_{\tau(1)}^kt_{\tau(2)}^{k-1}\cdots t_{\tau(k)})^{2^q-t}\cdot (\sum_{\phi\in \mathfrak{S}_{k-1}} t_{\phi(2)}^{k-1}\cdots t_{\phi(k)})^{t-2}\\ \cdot \sum_{\psi\in \mathfrak{S}_k} t_{\psi(1)}^{k-1}t_{\psi(2)}^{k-2}\cdots t_{\psi(k)}^0\cdot t_2 t_3^2\cdots t_k^{k-1}.   \end{align*}
Exponent considerations immediately give $\sigma(1)=\tau(k)=\psi(k)=1$ for all $\tau$, $\sigma(2)=\tau(k-1)=\phi(k)=\psi(k-1)=2$ for all $\tau$ and all $\phi$, and so on. Thus $p_U=t_1^d\cdots t_k^d$, so $F$ has a zero by Proposition~\ref{prop:Borsuk-Ulam}.\end{proof}

\section{Results for $\ell=3$ and $k=4$}
\label{sec:3}

For equipartitions by any 3 of 4 hyperplanes, the following theorems give part (b) of Theorems~\ref{thm:upper} and ~\ref{thm:upper perp}, respectively. 

\begin{theorem}
\label{thm:3} Let $q\geq 0$ and $1\leq t \leq 2^q$ be integers and let $d=7\cdot 2^q -2t$. For any $2^{q+1}-1$ masses on $\mathbb{R}^d$, there exists four hyperplanes $H_1,H_2,H_3, H_4$ such that 
\begin{compactenum}[(i)]
\item any three of $H_1, H_2, H_3, H_4$ equipartitions $2^{q+1}-t$ of the masses and
\item any two of $H_2, H_3, H_4$ equipartitions the remaining $t-1$ masses.
\end{compactenum}
\end{theorem}

\begin{theorem}
\label{thm:3 orthogonal} 
Let $q\geq 1$ and $2\leq t \leq 2^q$ be integers and let $d=7\cdot 2^q-2t$. Given any $2^{q+1}-2$ masses on $\mathbb{R}^d$, there exists $4$ pairwise orthogonal hyperplanes $H_1, H_2, H_3, H_4$ such that \begin{compactenum}[(i)]
\item any three of $H_1, H_2, H_3, H_4$ equipartition $2^{q+1}-t$ of the masses and 
\item  any two of $H_2,H_3, H_4$ equipartition the remaining $t-2$ masses. \end{compactenum}\end{theorem}

When $t=2$, one may strengthen Theorem~\ref{thm:3} as follows: 

\begin{proposition}
\label{prop:3} 
Let $q\geq 0$. Given $2^{q+1}-1$ masses on $\R^{7\cdot 2^q-4}$, there are four hyperplanes $H_1,H_2,H_3,H_4$ such that $2^{q+1}-2$ of the masses are equipartitioned by any three of the four hyperplanes and the remaining mass is equipartitioned by $H_2, H_3$, and $H_4$.
\end{proposition}

For example, given three masses $\mu_1,\mu_2,\mu_3$ on $\mathbb{R}^{10}$, then there are four hyperplanes such that any three of them equipartition the first two masses while three of them equipartition the third. This can be compared to the current estimates $7\leq \Delta(3;3)\leq 9$~\cite{MLVZ06} and $12\leq \Delta(3;4)\leq 17$~\cite{MLVZ06} for Question~\ref{quest:G-H-R}. 

\subsection{Computation of the polynomial $p_{3,4}$} 
\label{sec:poly 3}

In order to prove our $\ell=3$ results for Question~\ref{quest:generalized Makeev}, we first calculate $p_{3,4}=r_{1,4}\cdot r_{2,4}\cdot r_{3,4}$.

\begin{proposition} 
\label{prop:p_{3,4}}

$$p_{3,4}=\sum_{\sigma\in \mathfrak{S}_4}  t_{\sigma(1)}^5t_{\sigma(2)}^4t_{\sigma(3)}^3t_{\sigma(4)}^2 +
\sum_{\sigma\in\mathfrak{S}_4}t_{\sigma(1)}^7t_{\sigma(2)}^4t_{\sigma(3)}^2t_{\sigma(4)}^1+
\sum_{\sigma\in \mathfrak{S}_4}t_{\sigma(1)}^6t_{\sigma(2)}^5t_{\sigma(3)}^2t_{\sigma(4)}
+\sum_{\sigma\in \mathfrak{S}_4}t_{\sigma(1)}^6t_{\sigma(2)}^4t_{\sigma(3)}^3t_{\sigma(4)}.$$
\end{proposition}

\begin{proof}[Proof of Proposition~\ref{prop:p_{3,4}}] 

Considering $r_{3,4}=\prod_{1\leq i<j<k\leq 4} (t_i+t_j+t_k)$, let $s_j$ denote the $j$-th symmetric polynomial. We have 
$r_{3,4}=\prod_{1\leq j\leq 4}(t_j+s_1)=s_4+s_3s_1+s_2s_1^2+s_1s_1^3+s_1^4=s_4+s_1(s_3+s_2s_1).$
On the other hand, $s_2s_1=(t_1+t_2+t_3+t_4)\cdot \sum_{1\leq i<j\leq 4}t_it_j=\sum_{i\neq j}t_i^2t_j+3\sum_{1\leq i<j<k\leq 4}t_it_jt_k=\sum_{i\neq j}t_i^2t_j+s_3$. Thus $r_{3,4}= t_1t_2t_3t_4 +(t_1+t_2+t_3+t_4)\cdot \sum_{i\neq j} t_i^2t_j=t_1t_2t_3t_4+\sum_{i\neq j}t_i^3t_j$ because each term of the form $t_i^2t_j^2$ or $t_i^2t_jt_k$ occurs twice in $(t_1+t_2+t_3+t_4)\cdot \sum_{i\neq j} t_i^2t_j$. Thus $$p_{3,4}=\sum_{\sigma\in \mathfrak{S}_4}t_{\sigma(1)}^4t_{\sigma(2)}^3t_{\sigma(3)}^2t_{\sigma(4)}\cdot (\sum_{i\neq j} t_i^3 +t_1t_2t_3t_4).$$
As $\sum_\sigma t_{\sigma(1)}^{e_1}t_{\sigma(2)}^{e_2}t_{\sigma(3)}^{e_3}t_{\sigma(4)}^{e_4}=0$ whenever $e_i=e_j$ for $i\neq j$, an explicit examination of each product $\sum_\sigma t_{\sigma(1)}^4t_{\sigma(2)}^3t_{\sigma(3)}^2t_{\sigma(4)} \cdot t_i^3t_j$ shows that $p_{3,4}= 3\sum_\sigma t_{\sigma(1)}^5t_{\sigma(2)}^4t_{\sigma(3)}^3t_{\sigma(4)}^2 +
\sum_\sigma t_{\sigma(1)}^7t_{\sigma(2)}^4t_{\sigma(3)}^2t_{\sigma(4)}^1+
\sum_\sigma t_{\sigma(1)}^6t_{\sigma(2)}^5t_{\sigma(3)}^2t_{\sigma(4)}
+\sum_\sigma t_{\sigma(1)}^6t_{\sigma(2)}^4t_{\sigma(3)}^3t_{\sigma(4)}$, which gives Proposition~\ref{prop:p_{3,4}}.
\end{proof}

\subsection{Proof of Theorems~\ref{thm:3} and~\ref{thm:3 orthogonal}.} 
\label{sec: proof 3}

The proofs of Theorems~\ref{thm:3} and ~\ref{thm:3 orthogonal} are based on the following preliminary observation. Let $m=2^{q+1}-t$, where $1\leq t\leq 2^q$, and let $d=7\cdot 2^q-2t$. Let $h_1=\sum_\sigma t_{\sigma(1)}^5t_{\sigma(2)}^4t_{\sigma(3)}^3t_{\sigma(4)}^2$ and let $h$ denote the remaining sum in $p_{3,4}$. We have $p_{3,4}^m=(h_1^{2^q}+h^{2^q})\cdot p_{3,4}^{2^q-t}$, and it is easily seen that each monomial $t_1^{e_1}t_2^{e_2}t_3^{e_4}t_4^{e_4}$ of $h^{2^q}\cdot p_{3,4}^{2^q-t}$ has some exponent $e_i$ with $e_i\geq 6\cdot 2^q +2^q-t=7\cdot 2^q-t>d$. Thus $$p_{3,4}^m=h_1^m=t_1^mt_2^mt_3^mt_4^m\cdot p_{2,4}^m$$
in $\mathbb{Z}_2[t_1,t_2,t_3,t_4]/(t_1^{d+1},t_2^{d+1}, t_3^{d+1},t_4^{d+1})$. The proofs of Theorem~\ref{thm:3} and ~\ref{thm:3 orthogonal} then reduce to the same computations given in proofs of the $k=4$ case of Theorems~\ref{thm:2} and~\ref{thm:2 orthogonal}, respectively. 

\begin{proof}[Proof of Theorems~\ref{thm:3} and Theorems~\ref{thm:3 orthogonal}]

For Theorem~\ref{thm:3}, let $m=2^{q+1}-t$, where $1\leq t\leq 2^q$, and let $d=7\cdot 2^q -2t.$ As with Theorem~\ref{thm:2}, we shall further specify that $H_i\perp H_j$ for $(i,j)\in\{(1,3),(1,4),(2,4)\}$ and that each of $H_2,H_3,$ and $H_4$ bisects any further measure. Thus $U=U_{3,4}^{\oplus m} \oplus U_{2,3}^{\oplus{t-1}} \oplus \mathcal{O}(\mathcal{S})\oplus U_{1,k-1},$ where $\mathcal{S}=\{(1,3),(1,4),(2,4)\}$. By the discussion above, \begin{align*}
p_U=\sum_\sigma t_{\sigma(1)}^{4\cdot2^q}t_{\sigma(2)}^{3\cdot 2^q}t_{\sigma(3)}^{2\cdot 2^q}t_{\sigma(4)}^{2^q}\cdot (\sum_\tau t_{\tau(1)}^4t_{\tau(2)}^3t_{\tau(3)}^2t_{\tau(1)})^{2^q-t}\cdot (\sum_\phi t_{\phi(2)}^3t_{\phi(3)}^2t_{\phi(3)})^{t-1}\\\cdot (t_1+t_3)(t_1+t_4)(t_2+t_4)\cdot t_1^mt_2^{m+1}t_3^{m+1}t_4^{m+1}.
\end{align*}
The same exponent considerations as in the proof of the $k=4$ case of Theorem~\ref{thm:2} now show that $p_U=t_1^dt_2^dt_3^dt_4^d$, so any $\mathbb{Z}_2^4$-equivariant map $F\colon (S^d)^4\rightarrow U$ has a zero.

As with Theorem~\ref{thm:2 orthogonal}, for Theorem~\ref{thm:3 orthogonal} we impose the ``cascade'' condition that each $\mu_i$ of any $k-1$ additional masses $\mu_1,\ldots, \mu_{k-1}$ is bisected by each of the hyperplanes $H_{i+1},\ldots, H_k$ for all $1\leq i\leq k-1$. We then have  
\begin{align*} p_U=\sum_\sigma t_{\sigma(1)}^{4\cdot2^q}t_{\sigma(2)}^{3\cdot 2^q}t_{\sigma(3)}^{2\cdot 2^q}t_{\sigma(4)}^{2^q}\cdot (\sum_\tau t_{\tau(1)}^4t_{\tau(2)}^3t_{\tau(3)}^2t_{\tau(1)})^{2^q-t}\cdot (\sum_\phi t_{\phi(2)}^3t_{\phi(3)}^2t_{\phi(3)})^{t-2}\\\cdot (\sum_\psi t_{\psi(1)}^3t_{\psi(2)}^2t_{\psi(3)}^1t_{\psi(4)}^0)\cdot t_1^mt_2^{m+1}t_3^{m+2}t_4^{m+3},\end{align*}
so as with the $k=4$ case of  Theorem~\ref{thm:2 orthogonal} we get $p_U=t_1^dt_2^dt_3^dt_4^d$.\end{proof}

\subsection{Proof of Proposition~\ref{prop:3}} 
\label{sec:cascade}

The proof of Proposition~\ref{prop:3} involves a slightly more complicated calculation. As $p_{2,3}=\sum_{\sigma\in \mathfrak{S}_3}t_{\sigma(1)}^3t_{\sigma(2)}^2t_{\sigma(3)}^1$ it is easy seen that $p_{3,3}=p_{2,3}\cdot (t_1+t_2+t_3)=\sum_{\sigma\in\mathfrak{S}_3} t_{\sigma(1)}^4t_{\sigma(2)}^2t_{\sigma(3)}.$ 

\begin{proof}[Proof of Proposition~\ref{prop:3}] Let $q\geq 1$, let $m=2^{q+1}-2$, and let $d=7\cdot 2^q -4$. We consider the representation $U=U_{3,4}^{\oplus m}\oplus U_{3,3}\oplus U_{1,2}^{\oplus 2}\oplus U_{1,1}$, where as before  $U_{1,2}^{\oplus 2}\oplus U_{1,1}$ arises from a cascading bisection condition on $H_3$ and $H_4$ for three additional masses. As $p_{3,4}^m=(\sum_\sigma t_{\sigma(1)}^5t_{\sigma(2)}^4t_{\sigma(3)}^3t_{\sigma(4)}^2)^m$, we have 
$p_U=\sum_\sigma t_{\sigma(1)}^{5\cdot 2^q}t_{\sigma(2)}^{4\cdot 2^q}t_{\sigma(3)}^{3\cdot 2^q}t_{\sigma(4)}^{2\cdot 2^q}\cdot (\sum_\tau t_{\tau(1)}^5t_{\tau(2)}^4t_{\tau(3)}^3t_{\tau(4)}^2)^{2^q-2}\cdot \sum_\phi t_{\phi(2)}^4t_{\phi(3)}^2t_{\phi(4)}\cdot t_3^2t_4^3.$
Degree considerations force $\sigma(1)=\tau(4)=1$ for each $\tau$, so
$$p_U=t_1^d\cdot \sum_\sigma t_{\sigma(2)}^{4\cdot 2^q}t_{\sigma(3)}^{3\cdot 2^q}t_{\sigma(4)}^{2\cdot 2^q}\cdot (\sum_\tau t_{\tau(2)}^5t_{\tau(3)}^4t_{\tau(4)}^3)^{2^q-2}\cdot \sum_\phi t_{\phi(2)}^4t_{\phi(3)}^2t_{\phi(4)}\cdot t_3^2t_4^3.$$ 
We must have $\sigma(2)=2$, since otherwise $\deg(t_i)\geq 4\cdot 2^q+3\cdot (2^q-2)+1+2>d$.   Moreover, $\deg(t_2)\geq 7\cdot 2^q-12+8+1>d$ if $\tau(4)\neq 2$ for more than one of the $\tau$. We see therefore that $\deg(t_2)=d$ if and only if either (i) $\tau(4)=2$ for all $\tau$ and $\phi(3)=2$, or (provided $q>1$) that (ii) $\tau(4)=2$ for all but one $\tau$, $\tau(3)=2$ for exactly one $\tau$, and $\phi(4)=2$. As there are $2^q-2$ choices of $\tau$ in option (ii), this case contributes $(2^q-2)t_1^dt_2^d \cdot \sum_\sigma t_{\sigma(3)}^{3\cdot 2^q}t_{\sigma(4)}^{2\cdot 2^q}\cdot (\sum_\tau t_{\tau(3)}^5t_{\tau(4)}^4)^{2^q-3}\cdot \sum_\rho t_{\rho(3)}^5t_{\rho(4)}^3\cdot \sum_\phi t_{\phi(3)}^4t_{\phi(4)}^2\cdot t_3^2t_4^3=0$ to the sum. We therefore have
$$p_U=t_1^dt_2^d\cdot \sum_\sigma t_{\sigma(3)}^{3\cdot 2^q}t_{\sigma(4)}^{2\cdot 2^q}\cdot (\sum_\tau t_{\tau(3)}^5t_{\tau(4)}^4)^{2^q-2}\cdot \sum_\phi t_{\phi(3)}^4t_{\phi(4)}\cdot t_3^2t_4^3.$$ Considering $\deg(t_{\sigma(3)})$, we have as before that $\deg(t_{\sigma(3)})\geq 3\cdot 2^q+4\cdot (2^q-4)+10+1+2>d$ if $\tau(4)\neq \sigma(3)$ for more than one of the $\tau$. It follows that $\deg(t_{\sigma(3)})=d$ if and only if (i) $\tau(4)=\sigma(3)$ for all $\tau$ and $\phi(4)=\sigma(3)$ (in which case $\sigma(3)=4$), or else that (ii) $\tau(4)=\sigma(3)$ for all but one of the $\tau$, $\tau(3)=\sigma(4)$ for exactly one $\tau$, and $\phi(4)=\sigma(4)$ (in which case $\sigma(3)=3$). Option (ii) contributes $(2^q-2)\cdot t_1^dt_2^dt_3^dt_4^d=0$ to the sum, while option (i) gives $t_1^dt_2^dt_3^dt_4^d$, so $p_U=t_1^dt_2^dt_3^dt_4^d$ as desired.
\end{proof}

\section{Transversal Extensions}
\label{sec:transversal} 

As in~\cite{FMSS22}, our transversal generalization of Theorem~\ref{thm:upper} is ultimately derived from a topological Radon-type intersection theorem for which the configuration-space/test-map scheme essentially reduces to that used for Question~\ref{quest:generalized Makeev} and to which we may again apply ~Proposition~\ref{prop:Borsuk-Ulam}.

\subsection{Hyperplane Transversals via Radon-type theorems}

Let $\Delta_{n-1}$ denote the (geometric realization) of the $(n-1)$-dimensional simplex on the vertex set $[n]=\{1,2,\ldots, n\}$. By a Radon-type intersection theorem, we mean a result which prescribes conditions under which any continuous map $f\colon \Delta_{n-1}\rightarrow \R^d$ admits pairs $(\sigma^+,\sigma^-)$ of disjoint faces of $\Delta_{n-1}$ with $f(\sigma^+)\cap f(\sigma^-)=\emptyset$. As in ~\cite{FMSS22}, we call such faces \textit{Radon pairs}. For example, the topological Radon theorem ~\cite{BB79} shows that any continuous map $f\colon \Delta_{n-1}\rightarrow \R^d$ is guaranteed at least one Radon pair provided $n\geq d+2$. By restricting to \textit{linear} mappings (i.e., those sending convex sums of vertices to the convex sum of their images), this result gives the classical Radon's theorem~\cite{Ra21}: any finite set of at least $d+2$ points in $\R^d$ can be partitioned into disjoint subsets whose convex hulls intersect. 

Before stating the relevant Radon-type theorem for our context, it will be convenient to utilize terminology from the theory of hypergraphs. To that end, let $\F$ be a family of non-empty sets and $r\ge 2$  an integer. The $r$-uniform  Kneser hypergraph   $\KG^r(\F)$ of~$\F$ is defined by letting the vertices be the sets of ~$\F$ and letting the hyperedges consist of all collections $\{A_1, \dots, A_r\}$ of  pairwise disjoint sets of $\F$. The chromatic number ~$\chi(\KG^r(\F))$ is the minimum number of colors required so that the vertices of $\KG^r(\F)$ can all be colored without a monochromatic hyperedge. Thus $\chi(\KG^r(\F)) \leq m$ means precisely that $\F$ can be partitioned into $m$ families $\F = \F_1 \cup \dots \cup \F_m$ such that no $r$ sets from any $\F_i$ are pairwise disjoint.  

 By a minimal Radon pair $(\sigma^+, \sigma^-)$ for $f\colon \Delta_{n-1} \to \R^d$, we mean one such that $f(\sigma^\pm ) \cap f(\tau) = \emptyset$ for any proper subface $\tau \subset \sigma^\pm$. The following connection between linear Radon-type results, hyperplane transversals, and hyperplane equipartitions of finite point sets was given in ~\cite[Theorem 3.3.]{FMSS22}:

\begin{theorem}
\label{thm:implications}
Let $c \ge 1$, $d \ge 1$, $k\ge 1$, and $m\ge 1$ be integers. Of the statements below, (\ref{it:radon2}) implies~(\ref{it:affine-piercing}) and ~(\ref{it:affine-piercing}) implies (\ref{it:equipart}).
    \begin{compactenum} 
        \item\label{it:radon2} If $n\ge d+c+2$, then for any linear map $f\colon \Delta_{n-1}\to \mathbb{R}^d$ and for any family $\F$ of subsets of $[n]$ with $\chi(\KG^{2^k}(\F)) \le m$, there are $k$ minimal Radon pairs $(\sigma_1^+,\sigma_1^-),\ldots, (\sigma_k^+,\sigma_k^-)$ for~$f$ such that none of the
        intersections of the form $\sigma_1^\pm \cap \dots \cap \sigma_k^\pm$ contains a set from~$\F$.
        \item\label{it:affine-piercing} For any finite family $\F$ of polytopes in $\mathbb{R}^{c}$ with $\chi(\KG^{2^k}(\F)) \le m$, there are $k$ hyperplanes which pierce $\F$.
        \item\label{it:equipart} For finite point sets $X_1, \dots, X_m \subset \R^c$, there are $k$ hyperplanes $H_1, \dots, H_k$ such that each of the 
        regions $H_1^\pm \cap \dots \cap H_k^\pm$ contains no more than $\frac{1}{2^k}|X_i|$ points of each~$X_i$. 
    \end{compactenum} 
\end{theorem}

If one replaces $\chi(\KG^{2^k})$ with $\chi(\KG^{2^\ell})$ and also replaces condition ~(\ref{it:radon2}) with the condition that none of the intersections $\sigma_{i_1}^\pm \cap \dots \cap \sigma_{i_\ell}^\pm$ contains a set from $~\F$ for any choice of $1\leq i_1<i_2<\cdots i_\ell\leq k$, the identical proof to that given in ~\cite[Section 3]{FMSS22} immediately yields the following Makeev-type version of Theorem~\ref{thm:implications}:

\begin{proposition}
\label{prop:implications}
    Let $c \ge 1$, $d \ge 1$, $k\ge 1$, $1\leq \ell \leq k$, and $m\ge 1$ be integers. Of the statements below, (\ref{it:radon}) implies~(\ref{it:transverse}) and (\ref{it:transverse}) implies (\ref{it:makeev}).
    \begin{compactenum} 
        \item\label{it:radon} If $n\ge d+c+2$, then for any linear map $f\colon \Delta_{n-1}\to \mathbb{R}^d$ and for any family $\F$ of subsets of $[n]$ with $\chi(\KG^{2^\ell}(\F)) \le m$, there are $k$ minimal Radon pairs $(\sigma_1^+,\sigma_1^-),\ldots, (\sigma_k^+,\sigma_k^-)$ for~$f$ such that  none of the 
        intersections of the form $\sigma_{i_1}^\pm \cap \dots \cap \sigma_{i_\ell}^\pm$ contains a set from~$\F$ for any choice of $1\leq i_1<i_2<\cdots<i_\ell\leq k$. 
        \item\label{it:transverse} For any finite family $\F$ of polytopes in $\mathbb{R}^{c}$ with $\chi(\KG^{2^\ell}(\F)) \le m$, there are $k$ hyperplanes, any $\ell$ of which pierce $\F$.
        \item\label{it:makeev} For finite point sets $X_1, \dots, X_m \subset \R^c$, there are $k$ hyperplanes $H_1, \dots, H_k$ such that each
        region of the form $H_{i_1}^\pm \cap \dots \cap H_{i_\ell}^\pm$ contains no more than $\frac{1}{2^\ell}|X_i|$ points of each~$X_i$ for any choice of $1\leq i_1<i_2<\cdots < i_\ell\leq k$.
    \end{compactenum} 
\end{proposition}

We now state our desired Radon-type result, from which Theorem~\ref{thm:transversal} follows via Proposition~\ref{prop:implications}.

\begin{proposition}
\label{prop:Radon}   Let $m=2^{q+1}-t$, where $d\geq 1$, $q\geq 0$, and $1\leq t\leq 2^q$ are integers.
Suppose that $n\ge d+U(m;\ell,k)+1$, where
\begin{compactenum}[(a)]
    \item $U(m;\ell,k)=2^q\cdot (k+1)-t$ if 
 $\ell=2$ and
 \item $U(m;\ell,k)=7\cdot 2^q -2t$ if $(\ell,k)=(3,4)$.
\end{compactenum} 
If $\F$ is a family of subsets of $[n+1]$ with $\chi(\KG^{2^\ell}(\F)) \le m$, then for any continuous map $f\colon\Delta_n\rightarrow \R^d$ there exist $k$ Radon pairs $(\sigma_1^+, \sigma_1^{-}), \dots, (\sigma_k^+, \sigma^{-}_k)$ for~$f$ such that no intersection of the form $\sigma_{i_1}^\pm \cap \dots \cap \sigma_{i_\ell}^\pm$ contains a set from~$\F$ for any $1\leq i_1<i_2<\ldots<i_\ell \leq k$. 
\end{proposition}

\begin{proof}[Proof of Theorem~\ref{thm:transversal}] 
Let $d$ be as in the statement of Theorem~\ref{thm:transversal}. As any Radon pair can be made minimal by the removal of unnecessary vertices, Proposition~\ref{prop:Radon} together with Proposition~\ref{prop:implications} shows that any $m$ finite families $\F_1,\ldots, \F_m$ of compact convex sets in $\R^d$ can be pierced by any $\ell$ of some collection of $k$ hyperplanes provided no $2^\ell$ polytopes from any family are pairwise disjoint. Approximating arbitrary compact convex sets with polytopes finishes the proof.
\end{proof} 

As arbitrary masses can be approximated by finite point sets, we likewise have that Proposition~\ref{prop:implications} together with Theorem~\ref{thm:transversal} recovers Theorem~\ref{thm:upper}. 

\subsection{Proof of Proposition ~\ref{prop:Radon}} 

In order to prove Proposition ~\ref{prop:Radon}, we follow the configuration-space/test-map paradigm for multiple Radon pairs as given in ~\cite[Chapter 4]{FMSS22}. As in Section~\ref{sec:2}, we will define our test-maps using a Fourier perspective. 

\subsubsection{Configuration Spaces} As usual in Radon-type intersection theory, we consider the deleted join of $\Delta_n$ defined as follows. Let $\Delta_n^{\ast 2} =\{\sigma \ast \tau\mid \sigma,\tau\subseteq \Delta_n\}$ be the join of the $n$-simplex with itself. The deleted join $\Sigma_n:=(\Delta_n)^{\ast 2}_\Delta=\{\sigma^+\ast\sigma^-\mid \sigma^+\cap\sigma^-=\emptyset\}$ is the sub-complex of $\Delta_n^{\ast 2}$ consisting of all formal convex sums of the form $\lambda x^+ + (1-\lambda) x^-$ where $x^+ \in \sigma^+$, $x^-\in \sigma^-$, and $\sigma^+$ and $\sigma^-$ are disjoint. There is a free $\mathbb{Z}_2$-action on $\Sigma_n$ given by interchanging each $\lambda x^+ + (1-\lambda) x^-$ with $(1-\lambda)x^- +\lambda x^+$. As a complex, $\Sigma_n$ can be realized as the boundary of the $(n+1)$-dimensional cross--polytope, which is a topological $n$-dimensional sphere, and under this identification the action described above corresponds to the antipodal action on the sphere.

For any $k\ge 2$, we consider the $k$-fold product $$\Sigma_n^{\times k}=\{\sigma=(\sigma_1^+\ast \sigma_1^-,\cdots, \sigma_k^+\ast \sigma_k^-)\mid \sigma_i^+\cap \sigma_i^-=\emptyset\,\,\text{for all}\,\, 1\leq i \leq k\}$$ of the deleted join. There is then a free $\mathbb{Z}_2^k$ action on $\Sigma_n^{\times k}$ given by interchanging each pair of disjoint faces. Observing that the $(n+1)$-dimensional cross-polytope is the unit sphere in the $\ell_1$-metric on $\R^{n+1}$ and taking the product metric on $\Sigma_n^{\times k}$ with the $\ell_1$-metric on each $\Sigma_n$ factor, we see that this action is by isometries. Extending the identification of $\Sigma_n$ with $S^n$ thereby gives a natural equivariant correspondence between $(S^n)^k$ and $\Sigma_n^k$.

\subsubsection{Test Spaces and Test Maps} 

Let $n=d+c+1$ for some constant $c$. To obtain $k$ Radon pairs for a given continuous map $f:\Delta_n\rightarrow \R^d$, we define $R\colon\Sigma_n^{\times k}\rightarrow U_{1,k}^{\oplus (d+1)}$ by $$R(\lambda x)=(\lambda_1-1/2,\lambda_1f(x_1^+)-(1-\lambda_1)f(x_1^-),\ldots,\lambda_k-1/2, \lambda_k f(x_k^+)-(1-\lambda_k)f(x_k^-))$$  for each $\lambda x\colon =(\lambda_1x_1^++(1-\lambda_1)x_1^-,\ldots,\lambda_kx_k^++(1-\lambda_k)x_k^-)\in \Sigma_n^{\times k}$. It is easily confirmed that $R$ is $\mathbb{Z}_2^k$-equivariant. 

Now suppose that $\F$ is a family of subsets of the vertex set of $\Delta_n$ such that~$\chi(\KG^{2^\ell}(\F)) \le m$ and consider a partition $\F=\F_1 \cup \dots \cup \F_m$ such that no $2^\ell$ subsets of any of the $\F_i$ are pairwise disjoint. For each $\sigma=(\sigma_1^+\ast \sigma_1^-,\ldots, \sigma_k^+\ast\sigma_k^-)\in \Sigma_n^{\times k}$, we write $\sigma_i^0$ for $\sigma_i^+$ and $\sigma_i^1$ for~$\sigma_i^-$. For any $1\leq j\leq m$ and any $g=(g_1,\ldots, g_k)\in \mathbb{Z}_2^k$, we let $$K^j_{g}=\left\{\sigma\in \Sigma_n^{\times k}\mid \sigma_1^{g_1}\cap\cdots \cap \sigma_k^{g_k} \notin \F_j \right\}.$$ Without loss of generality, we may assume that each $\F_j$ is closed under taking supersets and therefore that each $K^j_g$ is a simplicial complex. 

As in Section~\ref{sec:topology}, we let $E_{\ell,k}=\{h=(h_1,\ldots, h_k)\in \mathbb{Z}_2^k \mid 1\leq h_1+\cdots +h_k\leq \ell\}$. For any $1\leq i_1<i_2<\cdots <i_\ell<k$, we again let $\mathbb{Z}_2^\ell=\langle \mathbf{e}_{i_1},\ldots, \mathbf{e}_{i_\ell}\rangle$ and let $\mathbb{Z}_2^{k-\ell}$ be its orthogonal complement. For each $g\in \mathbb{Z}_2^\ell$, let $$L_g^j=\{\sigma\in \Sigma_n^{\times k}\mid \sigma_{i_1}^{g_1}\cap \cdots \cap \sigma_{i_\ell}^{g_\ell} \notin \F_j\}.$$  As $L_g^j=\cap_{g'\in \mathbb{Z}_2^{k-\ell}}K_{g+g'}^j$, each $L_g^j$ is a simplicial complex, as is  $$L^j:=\cap_{g\in E_{\ell,k}} L_g^j.$$ By construction, the complex $L=\cap_{j=1}^m L^j$ thus consists of all $\lambda x\in \sigma$ such that $\sigma_{i_1}^\pm \cap \cdots \sigma_{i_\ell}^\pm\notin \F$ for any choice of $1\leq i_1<\ldots<i_\ell\leq k$. 

For each $\lambda x\in \Sigma_n^{\times k}$, consider the distance maps $d_j\colon \mathbb{Z}_2^k\rightarrow \R$ given by $$d_j(g)=d(\lambda x, K_g^j)$$ and their Fourier expansions $$d_j=\sum_{h\in \mathbb{Z}_2^k}c_h(d_j)\chi_h.$$ One has the following analogue of Proposition~\ref{prop:Fourier}:

\begin{proposition}
\label{prop: Fourier Radon} Let $\lambda x\in \Sigma_n^{\times k}$. Then $\lambda x\in L$ if and only if $c_h(d_j)=0$ for all $h\in E_{\ell,k}$ and all $1\leq j\leq m$. 
\end{proposition}

\begin{proof}[Proof of Proposition~\ref{prop: Fourier Radon}]
Let $\mathbb{Z}_2^\ell$ be the subgroup corresponding to a choice of $1\leq i_1<\cdots <i_\ell\leq k$ and  let $\mathbb{Z}_2^{k-\ell}$ be its complement. For each $1\leq j \leq m$, define $d_j^\ell\colon \mathbb{Z}_2^\ell \rightarrow \R$ by $$d_j^\ell(g)=\sum_{g'\in \mathbb{Z}_2^{k-\ell}}d_j(g+g').$$ For any $g\in \mathbb{Z}_2^\ell$ we have $L_g^j=\cap_{g'\in \mathbb{Z}_2^{k-\ell}}K_{g+g'}^j$, so $\lambda x\in L_g^j$ if and only if  $d_j^\ell(g)=0$. Thus $\lambda x\in L^j$ if and only if $d_j^\ell$ is the zero map for each choice of $\mathbb{Z}_2^\ell$. 

As in the proof of Proposition~\ref{prop:Fourier}, consider the Fourier expansion $$d_j^\ell=\sum_{h\in \mathbb{Z}_2^\ell} c_h(d_j^\ell)\chi_h{_{\mid \mathbb{Z}_2^\ell}}.$$ As $d_j^\ell(g)=\sum_{g'\in \mathbb{Z}_2^{k-\ell}}d_j(g+g')$ for each $g\in \mathbb{Z}_2^\ell$, the same argument as in the proof of Proposition~\ref{prop:Fourier} shows that $c_h(d_j^\ell)=2^{\ell-k}c_h(d_j)$. Thus each $d_j^\ell$ is a constant map if and only if $c_h(d_j)=0$ for all $h\in E_{\ell,k}$. We now show that any $d_j^\ell$ must vanish on some $g(j)\in \mathbb{Z}_2^\ell$, which completes the proof. To see this, consider the various intersections $\sigma_{i_1}^\pm\cap\cdots \cap \sigma_{i_\ell}^\pm$, which are pairwise disjoint. If $\lambda x\notin L_g^j$ for all $g\in \mathbb{Z}_2^\ell$, then each of these intersections would lie in the family $\F_j$. As $\F_j$ only contains non-empty sets, this would mean that  $\F_j$ contains $2^\ell$ pairwise disjoint sets, a contradiction. Thus $d_j^\ell(g(j))=0$ for some $g(j)\in \mathbb{Z}_2^\ell$. 
\end{proof} 

As with Question~\ref{quest:generalized Makeev}, we define $\mathcal{D}_{j,h}\colon \Sigma_n^{\times k}\rightarrow U_{\ell,k}$ by $$\mathcal{D}_{j,h}(\lambda x) =\sum_{g\in \mathbb{Z}_2^k}d(\lambda x, K_g^j)\chi_h(g)$$ for each $h\in E_{\ell,k}$ and each $1\leq j\leq m$. For each $\sigma\in \Sigma_n^{\times k}$, we have $g \sigma \cap \sigma=\emptyset$ for all  $g\in \Z_2^k\setminus\{0\}$. As the $\mathbb{Z}_2^k$-action on $\Sigma_n^{\times k}$ is by isometries, it is a straightforward exercise to verify that each $\mathcal{D}_{j,h}$ is $\mathbb{Z}_2^k$-equivariant, and therefore so is $$\mathcal{D}=\oplus_{j=1}^m \oplus_{h\in E_{\ell,k}}\mathcal{D}_{j,h}\colon \Sigma_n^{\times k}\rightarrow U_{\ell,k}^{\oplus m}.$$ 

We now give our test map for Proposition~\ref{prop:Radon}. Let $\iota\colon (S^n)^k\rightarrow \Sigma_n^{\times k}$ be the natural equivariant identification of $(S^n)^k$ with $\Sigma_n^{\times k}$ and let \begin{equation}\label{eqn:F'} F' = (\mathcal{D}\oplus R)\circ \iota\colon (S^n)^k \rightarrow U_{\ell,k}^{\oplus m} \oplus U_{1,k}^{\oplus (d+1)}. \end{equation} By construction, the zeros of $F'$ correspond precisely to those $k$ Radon pairs $(\sigma_1^+,\sigma_1^-),\ldots, (\sigma_k^+,\sigma_k^-)$ for which no intersection $\sigma_{i_1}^\pm\cap \cdots \cap \sigma_{i_\ell}^\pm$ contains a set of ~$\F$ for any choice of $1\leq i_1<\cdots<i_\ell\leq k$. Thus Proposition~\ref{prop:Radon} follows if any $\mathbb{Z}_2^k$-map $F'$ can be guaranteed a zero when $c\geq U(m;\ell,k)$. 

\begin{proof}[Proof of Proposition~\ref{prop:Radon}]
Let $n=d+U(m;\ell,k)+1$ and let $$U=U_{\ell,k}^{\oplus m}\oplus U_{1,k}^{\oplus (d+1)}\oplus U_{1,k-1}^{\oplus t}\oplus U_{1,k-2}^{\oplus t}\oplus \cdots \oplus U_{1,1}^{\oplus t}.$$ We show that any equivariant map $F\colon (S^n)^k\rightarrow U$ has a zero via Proposition~\ref{prop:Borsuk-Ulam}, and in particular so must the map $F'\colon (S^n)^k\rightarrow U_{\ell,k}^{\oplus m} \oplus U_{1,k}^{\oplus (d+1)}$ above.  

When $\ell=2$ and $k$ is arbitrary, $U(m;2,k)=(k+1)\cdot 2^q -t$. We have 
$p_U=\sum_{\sigma\in \mathfrak{S}_k}t_{\sigma(1)}^{2^q k}t_{\sigma(2)}^{2^q(k-1)} \cdots t_{\sigma(k)}^{2^q}\cdot (\sum_{\tau\in \mathfrak{S}_k}t_{\tau(1)}^kt_{\tau(2)}^{k-1}\cdots t_{\tau(k)})^{2^q-t}\cdot t_1^{d+1}t_2^{d+t+1}t_3^{d+t+2}\cdots t_k^{d+t+k}$.  Degree considerations immediately imply that $\sigma(i)=i=\tau(k-i+1)$ for all $1\leq i \leq k$ and all $\tau$, so  $p_U=t_1^n\cdots t_k^n$. 

When $\ell=3$ and $k=4$, $U(m;3,4)=7\cdot 2^q-2t$. As $p_{3,4}^m=t_1^mt_2^mt_3^mt_4^m\cdot (\sum_\sigma t_{\sigma(1)}^4t_{\sigma(2)}^3t_{\sigma(3)}^2t_{\sigma(4)})^m$,  $p_U=(\sum_\sigma t_{\sigma(1)}^{4\cdot2^q}t_{\sigma(2)}^{3\cdot 2^q}t_{\sigma(3)}^{2\cdot 2^q}t_{\sigma(4)}^{2^q})\cdot (\sum_\tau t_{\tau(1)}^4t_{\tau(2)}^3t_{\tau(3)}^2t_{\tau(1)})^{2^q-t}\cdot t_1^{m+d+1}t_2^{m+d+2}\cdots t_k^{m+d+k}.$ It is again immediate that $p_U=t_1^nt_2^nt_3^nt_4^n$. This completes the proof. 
\end{proof} 

\section{Appendix} 

We conclude with a verification of the following estimates  mentioned in the introduction. 

\begin{proposition}
\label{prop:perp}$\newline$
\vspace*{-.2in}
\begin{compactenum}[(a)]
\item $5\leq \Delta^\perp(1;3/4)\leq 7$ and
\item $7\leq \Delta^\perp(1;3/5)\leq 9$
\end{compactenum} 
    \end{proposition}
 
\begin{proof}[Proof of Proposition~\ref{prop:perp}(a)] As in the proof of Proposition~\ref{prop:3}, to derive the upper bound $\Delta^\perp(1;3/4)\leq 7$ we augment the relevant representation and consider $U=U_{3,4}\oplus U(\mathcal{O})\oplus U_{1,3}\oplus U_{1,2}^{\oplus 2}\oplus U_{1,1}$. Again, we show that any continuous equivariant map $F\colon (S^7)^4\rightarrow U$ has a zero via Proposition~\ref{prop:Borsuk-Ulam}. 

As $p_{3,4}=t_1t_2t_3t_4\cdot r_{2,4}\cdot r_{3,4}$ and $p_{\mathcal{O}}=r_{2,4}$, we have $p_U=t_1t_2t_3t_4\cdot p_{2,4}^2\cdot r_{3,4}\cdot t_2t_3^3t_4^4$ and so
$p_U=\sum_\sigma t_{\sigma(1)}^7t_{\sigma(2)}^5t_{\sigma(3)}^3t_{\sigma(4)}^1\cdot r_{3,4}\cdot t_2t_3^3t_4^4.$
We have $\sigma(1)=1$ by degree considerations. As $r_{3,4}=\prod_{2\leq i<j\leq 4}(t_1+t_i+t_j)\cdot (t_2+t_3+t_4)$ and $t_1^7\cdot \prod_{2\leq i<j\leq 4}(t_1+t_i+t_j)=t_1^d \cdot r_{2,4}$, 
$$p_U=t_1^7\cdot \sum_\sigma t_{\sigma(2)}^6t_{\sigma(3)}^4t_{\sigma(4)}^2\cdot \sum_\tau t_{\tau(2)}^2 t_{\tau(3)}^1t_{\tau(4)}^0\cdot (t_2+t_3+t_4)\cdot t_3^2t_4^3.$$
\noindent By degree considerations it is immediate that $\sigma(2)=2$, and we have $\deg(t_2)=7$ only if either (1) $\tau(3)=2$ or (2) $\tau(4)=2$. As option (1) yields $\sum_\sigma t_{\sigma(3)}^4t_{\sigma(4)}^2\cdot \sum_\tau t_{\tau(3)}^2 t_{\tau(4)}^0\cdot t_3^2t_4^2=\sum_\sigma t_{\sigma(3)}^8t_{\sigma(4)}=0$, we therefore have
$p_U=t_1^7t_2^6\cdot \sum_\sigma t_{\sigma(3)}^6t_{\sigma(4)}^4\cdot \sum_\tau t_{\tau(3)}^2 t_{\tau(4)}^1\cdot (t_2+t_3+t_4)\cdot t_4$. Again degree considerations force $\sigma(3)=\tau(4)=3$, so $p_U=t_1^7t_2^7t_3^7t_4^7.$ \end{proof}

For the proof of $\Delta^\perp(1;3/5)\leq 9$, we use the full expansion of $p_{3,4}$ given by Proposition~\ref{prop:p_{3,4}}. 

\begin{proof}[Proof of Proposition~\ref{prop:perp}(b)]

We show that any equivariant map $F\colon (S^9)^4\rightarrow U\colon=U_{3,5}\oplus U(\mathcal{O})\oplus U_{1,4}\oplus U_{1,3}\oplus U_{1,2}\oplus U_{1,1}$ has a zero. As in the proof of Theorem~\ref{cor:explicit perp}(d), we have $p_U=t_1t_2t_3t_4t_5\cdot r_{2,5}^2\cdot r_{3,5}\cdot t_2t_3^2t_4^3t_5^4,$ so
$p_U=\sum_\sigma t_{\sigma(1)}^9t_{\sigma(2)}^7t_{\sigma(3)}^5t_{\sigma(4)}^3t_{\sigma(5)}^1\cdot r_{3,5}\cdot t_2t_3^2t_4^3t_5^4.$

Again, degree considerations force $\sigma(1)=1$. As $r_{3,5}=\prod_{2\leq i<j\leq 5}(t_1+t_i+t_j)\cdot r_{3,4}$, we have $t_1^9\cdot r_{3,5}=t_1^9\cdot r_{2,4}\cdot r_{3,4}.$ 
Thus $$p_U=t_1^9\sum_\sigma t_{\sigma(2)}^7t_{\sigma(3)}^5t_{\sigma(4)}^3t_{\sigma(5)}^1\cdot p_{3,4}\cdot t_3t_4^2t_5^3.$$
\noindent We now show that $\sum_\sigma t_{\sigma(2)}^7t_{\sigma(3)}^5t_{\sigma(4)}^3t_{\sigma(5)}^1\cdot p_{3,4}\cdot t_3t_4^2t_5^3=5t_2^9t_3^9t_4^9t_5^9$, which gives $p_U=t_1^9t_2^9t_3^9t_4^9t_5^9$ as desired. To that end, let $q=\sum_\sigma t_{\sigma(2)}^7t_{\sigma(3)}^5t_{\sigma(4)}^3t_{\sigma(5)}^1\cdot t_3t_4^2t_5^3$, and let $p_1=\sum_{\sigma}  t_{\sigma(2)}^5t_{\sigma(3)}^4t_{\sigma(4)}^3t_{\sigma(5)}^2$, $p_2=\sum_{\sigma}t_{\sigma(2)}^7t_{\sigma(3)}^4t_{\sigma(4)}^2t_{\sigma(5}^1$, $p_3=\sum_{\sigma}t_{\sigma(2)}^6t_{\sigma(3)}^5t_{\sigma(4)}^2t_{\sigma(5)}$, and $p_4=\sum_{\sigma}t_{\sigma(2)}^6t_{\sigma(3)}^4t_{\sigma(4)}^3t_{\sigma(5)}$ be the summands of $p_{3,4}=p_1+p_2+p_3+p_4$.\\
\\
$\bullet$ For $q\cdot p_1=\sum_\sigma t_{\sigma(2)}^7t_{\sigma(3)}^5t_{\sigma(4)}^3t_{\sigma(5)}^1\cdot \sum_\tau t_{\tau(2)}^5t_{\tau(3)}^4t_{\tau(4)}^3t_{\tau(5)}^2\cdot t_3t_4^2t_5^3$, degree considerations immediately force $\sigma(2)=\tau(5)=2$, $\sigma(3)=\tau(4)=3$, and $\sigma(4)=\tau(3)=4$, so $q\cdot p_1=t_2^9t_3^9t_4^9t_5^9$.\\
\\
$\bullet$ For $q\cdot p_2=\sum_\sigma t_{\sigma(2)}^7t_{\sigma(3)}^5t_{\sigma(4)}^3t_{\sigma(5)}^1\cdot\sum_{\tau}t_{\tau(2)}^7t_{\tau(3)}^4t_{\tau(4)}^2t_{\tau(5)}^1\cdot t_3t_4^2t_5^3$, we have $\tau(2)=\sigma(5)=3$ by degree considerations, so
$q\cdot p_2=t_3^9\cdot \sum_\sigma t_{\sigma(2)}^7t_{\sigma(4)}^5t_{\sigma(5)}^3\cdot\sum_{\tau}t_{\tau(2)}^4t_{\tau(4)}^2t_{\tau(5)}^1\cdot t_4^2t_5^3$. We now have $\sigma(2)=\tau(4)$, as otherwise $\deg(\sigma(2))\neq 9$, and so $\sigma(2)=2$. Thus $q\cdot p_2=t_2^9t_3^9\cdot \sum_\sigma t_{\sigma(4)}^5t_{\sigma(5)}^3\cdot\sum_{\tau}t_{\tau(4)}^4t_{\tau(5)}^1\cdot t_4^2t_5^3=t_2^9t_3^9t_4^9t_5^9$.\\
\\
$\bullet$ For $q\cdot p_3=\sum_\sigma t_{\sigma(2)}^7t_{\sigma(3)}^5t_{\sigma(4)}^3t_{\sigma(5)}^1\cdot\sum_{\tau}t_{\tau(2)}^6t_{\tau(3)}^5t_{\tau(4)}^2t_{\tau(5)}^1\cdot t_3t_4^2t_5^3$, we have either that (i) $\sigma(2)=\tau(4)=2$ or that (ii) $\sigma(2)=\tau(5)=3$. We show that both cases contribute $t_2^9t_3^9t_4^9t_5^9$. Case (i) gives $t_2^9\cdot \sum_\sigma t_{\sigma(3)}^5t_{\sigma(4)}^3t_{\sigma(5)}^1\cdot\sum_{\tau}t_{\tau(3)}^6t_{\tau(4)}^5t_{\tau(5)}^1\cdot t_3t_4^2t_5^3=t_2^9t_3^9t_4^9t_5^9$ because we must have $\sigma(3)=\tau(5)=5$ and $\sigma(4)=\tau(4)=3$. In case (ii), we must have we must have $\sigma(2)=\tau(5)=4$ in $t_3^9\cdot \sum_\sigma t_{\sigma(2)}^5t_{\sigma(4)}^3t_{\sigma(5)}^1\cdot\sum_{\tau}t_{\tau(2)}^6t_{\tau(4)}^5t_{\tau(5)}^2\cdot t_4^2t_5^3$. This gives $t_3^9t_4^9\cdot \sum_\sigma t_{\sigma(2)}^3t_{\sigma(5)}^1\cdot \sum_\tau t_{\tau(2)}^6t_{\tau(5)}^5\cdot t_5^3$. It now follows that $\sigma(2)=\tau(2)=2$, so again one has $t_2^9t_3^9t_4^9t_5^9$.\\
\\
$\bullet$ Finally, we must have $\sigma(2)=\tau(5)=3$ in $q\cdot p_4=\sum_\sigma t_{\sigma(2)}^7t_{\sigma(3)}^5t_{\sigma(4)}^3t_{\sigma(5)}^1\cdot\sum_{\tau}t_{\tau(2)}^6t_{\tau(3)}^4t_{\tau(4)}^3t_{\tau(5)}^1\cdot t_3t_4^2t_5^3$. Thus $q\cdot p_4=
t_3^9\cdot\sum_\sigma t_{\sigma(2)}^5t_{\sigma(4)}^3t_{\sigma(5)}^1\cdot\sum_{\tau}t_{\tau(2)}^6t_{\tau(4)}^4t_{\tau(5)}^3 \cdot t_4^2t_5^3$, and so we must have $\sigma(2)=\tau(4)=2$. This gives $q\cdot p_4=t_2^9t_3^9\cdot \sum_\sigma t_{\sigma(4)}^3t_{\sigma(5)}^1\cdot\sum_{\tau}t_{\tau(4)}^6t_{\tau(5)}^3 \cdot t_4^2t_5^3$. Continuing we have $\sigma(5)=\tau(4)=4$ and so $q\cdot p_4=t_2^9t_3^9t_4^9t_5^9$. \end{proof}

\section{Acknowledgements}

The authors are grateful to Michael Crabb, from whom we learnt the proof of Proposition~\ref{prop:lower}. The authors also thank Florian Frick for helpful comments.

\end{document}